\documentclass[10pt]{amsart} 

\usepackage[psamsfonts]{amssymb} 
\usepackage{amsfonts,amsmath} 
\usepackage{graphicx, color}
\usepackage{enumerate}
\usepackage{url}
\usepackage{hyperref}

\usepackage[all]{xy}

\newtheorem{thmA}{Theorem}

\newtheorem{theorem}{Theorem}[section]

\newtheorem{proposition}[theorem]{Proposition}
\newtheorem{lemma}[theorem]{Lemma}
\newtheorem{corollary}[theorem] {Corollary}

\newtheorem{question}[theorem]{Question}
\newtheorem*{claim*}{Claim}

\theoremstyle{remark}
\newtheorem{remark}[theorem]{Remark}

\theoremstyle{definition}
\newtheorem{definition}[theorem]{Definition}

\def\calg{\mathcal{G}}
\def\calh{\mathcal{H}}
\def\calj{\mathcal{J}}

\def\Z{\mathbb Z}

\def\G{\Gamma}

\def\stab{{\rm{Stab}}}

\def\inn{{\rm{Inn}}}

\def\aut{{\rm{Aut}}}
\def\out{{\rm{Out}}}
\def\outgh{{\out(A_\Gamma;\calg,\calh^t)}}
\def\outogh{{\out^0(A_\Gamma;\calg,\calh^t)}}

\def\outghp{{\out^{[3]}(A_\Gamma;\calg,\calh^t)}}
\newcommand{\outghv}[1][v]{\out^{0}(\G_{#1};\calg_{#1},\calh_{#1}^t)}

\def\gln{{\rm{GL}}(n, \Z)}

\def\glm{{\rm{GL}}(m, \Z)}

\def\FR{{{\rm{FR}}_\mathcal{G}}}
\def\FRl{{{\rm{FR}}^{[3]}_\mathcal{G}}}

\def\gl{{\rm{GL}}}

\def\vcd{{\rm{vcd}}}
\def\cd{{\rm{cd}}}
\def\gd{{\rm{gd}}}

\def\<{\langle}
\def\>{\rangle}

\newcommand{\st}{\mathrm{st}}
\newcommand{\lk}{\mathrm{lk}}

\renewcommand{\setminus}{-}

\definecolor{olive}{rgb}{0,0.5,0}

\title[Calculating the vcd of $\aut(A_\G)$.]{Calculating the virtual cohomological dimension of the automorphism group of a RAAG.} 
\author{Matthew B. Day, Andrew W. Sale and Richard D. Wade}   

\begin{document}

\begin{abstract}
We describe an algorithm to find the virtual cohomological dimension of the automorphism group of a right-angled Artin group. The algorithm works in the relative setting; in particular it also applies to untwisted automorphism groups and basis-conjugating automorphism groups. The main new tool is the construction of free abelian subgroups of certain Fouxe-Rabinovitch groups of rank equal to their virtual cohomological dimension, generalizing a result of Meucci in the setting of free groups. 
\end{abstract}

\maketitle

\section{Introduction}
Automorphism groups of right-angled Artin groups (or RAAGs) form a diverse and interesting family, encompassing the rich worlds of both integer matrix groups and automorphism groups of free groups. For any right-angled Artin group $A_\G$, (determined by a finite graph $\G$) Laurence \cite{Laurence} gave a generating set for $\aut(A_\G)$, and since this result authors have worked to understand higher finiteness properties of these groups. In particular, Charney and Vogtmann  \cite{CVFiniteness} showed that each outer automorphism group $\out(A_\G)$ has finite \emph{virtual cohomological dimension (vcd)}. Given recent constructions of classifying spaces for \emph{untwisted  subgroups} \cite{MR3626599} and the analogs of \emph{congruence kernels} for these groups \cite{DW2}, it is natural to ask what $\vcd(\out(A_\G))$ actually is. Indeed, upper and lower bounds for specific examples and interesting subfamilies have been obtained in many cases \cite{BCV,MR3626599, DW2, MV}, giving the vcd when these bounds coincide.

In this paper we give an algorithm to compute the virtual cohomological dimension of $\out(A_\G)$ for an arbitrary graph $\G$. More generally, this algorithm gives the virtual cohomological dimension of any outer automorphism group of a right-angled Artin group relative to a collection of special subgroups. This includes the untwisted automorphism groups of \cite{MR3626599} and partially symmetric (or basis-conjugating) outer automorphism groups of RAAGs.

%

The relative (outer) automorphism groups mentioned above were studied extensively in \cite{DW2}, and are affectionately known as RORGs. Such a group is defined by taking collections $\calg$, $\calh$ of \emph{special subgroups} (a special subgroup is one of the form $A_\Delta$ given by an induced
 subgraph $\Delta \subset \Gamma)$  of a right-angled Artin group $A_\G$ and looking at the subgroup $\out(A_\G;\calg,\calh^t)$ of automorphisms that \emph{preserve} each element of $\calg$ and \emph{act trivially} on each element of $\calh$ (see Section~\ref{sec:raag-background}). This approach is not an idle exercise in generalization if one wants to understand $\out(A_\G)$. The main result of \cite{DW2} uses RORGs to construct a subnormal series for $\out(A_\G)$ (more generally, for an arbitrary RORG) such that the consecutive quotients of this series are either finite, free-abelian groups, copies of $\gl(n,\mathbb{Z})$, or groups known as \emph{Fouxe-Rabinovitch groups}. We call such a normal series a \emph{decomposition series}. In \cite{DW2}, decomposition series were used to iteratively construct finite classifying spaces for congruence subgroups of $\out(A_\G)$. In a similar fashion, we will see that the virtual cohomological dimension of $\out(A_\G)$ is the sum of the vcds of the consecutive quotients appearing in a decomposition series.

To make this process algorithmic, one needs to know how to find the vcd of a Fouxe-Rabinovitch group. Let us first recall the definition of these groups. Let \[ G=G_1 \ast G_2 \ast \cdots G_k \ast F_m\] be a free factor decomposition of a group $G$. An element $\Phi \in \out(G)$ belongs to the \emph{Fouxe-Rabinovitch} group associated to this free factor decomposition if for each $G_i$ there exists a representative $\phi \in \Phi$ restricting to the identity on $G_i$. For example, the basis-conjugating automorphism group of a free group is the Fouxe-Rabinovitch group given by the free factor decomposition $F_n=\mathbb{Z} \ast \mathbb{Z} \ast \cdots \ast \mathbb{Z}$. Going back to RAAGs, if each $G_i=A_{\Delta_i}$ is a special subgroup, then the Fouxe-Rabinovitch group is the relative automorphism group $\out(A_\G;\{A_{\Delta_i}\}^t)$.

\begin{thmA}\label{th:vcdfr}
Let $A_\Gamma=A_{\Delta_1} \ast A_{\Delta_2} \ast \cdots \ast A_{\Delta_k} \ast F_m$ be a free factor decomposition of a  right-angled Artin group with $k\geq 1$. Let $d(\Delta_i)$ be the size of a maximal clique in each $\Delta_i$, and let $z(\Delta_i)$ be the rank of the center of $A_{\Delta_i}$. Then \[\vcd(\out(A_\G;\{A_{\Delta_i}\}^t))=(k+2m-2)\cdot\max_i\{d(\Delta_i)\} + \sum_{i=1}^k (d(\Delta_i)-z(\Delta_i)). \]There exists a free abelian subgroup of $\out(A_\G;\{A_{\Delta_i}\}^t)$  of rank equal to the virtual cohomological dimension.
\end{thmA}

This generalizes a theorem of Meucci \cite{Meucci} on relative automorphism groups of free groups and Collins \cite{Collins} on partially symmetric automorphism groups. To prove this theorem, we obtain a lower bound for the virtual cohomological dimension by constructing free-abelian subgroups of the appropriate rank. The upper bound is obtained by a careful analysis of simplex stabilizers for the action of the Fouxe-Rabinovitch group on the spine of Guirardel and Levitt's \emph{relative Outer space} \cite{GL}. Roughly speaking, we have to make sure that simplices of large dimension have small stabilizers. 

In the case where $k=0$, the virtual cohomological dimension of $\out(F_m)$ was shown to be $2m-3$  in Culler and Vogtmann's seminal paper on Outer space \cite{CuV}. There, the lower bound is obtained by finding a copy of $\mathbb{Z}^{2m-3}$ in $\out(F_m)$ generated by Nielsen automorphisms. The abelian subgroups found in the Fouxe-Rabinovitch case are very similar and made up of transvections and partial conjugations (see Remark~\ref{remark:abelian_description}). On the other side of the RAAG spectrum, similar results hold for $\gl(n,\mathbb{Z})$. Here the virtual cohomological dimension is equal to the Hirsch length of the polycyclic subgroup of upper triangular matrices. Given all of this, it is natural to conjecture that for an arbitrary RORG there is also a polycyclic subgroup of rank equal to the virtual cohomological dimension. Indeed, this conjecture holds in all known examples, but we cannot prove it in general. Luckily, we do not need explicit polycyclic subgroups to calculate vcd.
 
\begin{thmA}\label{th:vcdrorg}
There is an algorithm which, given the input of a finite graph $\G$ and two collections of special subgroups $\calg$ and $\calh$ of $A_\G$, computes the virtual cohomological dimension of $\outgh$.
\end{thmA}

If one prefers to look at the absolute automorphism group, the vcds of $\aut(A_\G)$ and $\out(A_\G)$ differ by the dimension of $A_\G/Z(A_\G)$ (see Remark~\ref{rem:aut-out}).

The main idea behind the proof of Theorem~\ref{th:vcdrorg} is as follows.  Although virtual cohomological dimension is only subadditive with respect to exact sequences, \emph{rational cohomological dimension} is additive (by a theorem of Bieri \cite{Bieri}). It is not generally the case that the rational cohomological dimension of a torsion-free group agrees with its (integral) cohomological dimension (an example where they differ is given in the introduction of \cite{BM}). Luckily for us, the existence of the polycyclic subgroups above imply that for every consecutive quotient in a decomposition series, rational cohomological dimension coincides with vcd.  Therefore the vcd of a RORG is the sum of the vcds of the pieces that appear in its decomposition series, which we can find either from previous work in the literature \cite{BorelSerre, CuV}, or from Theorem~\ref{th:vcdfr}. 

\medskip

\paragraph*{\textbf{Related problems.}} The groups $\gl_n(\mathbb{Z})$ and $\out(F_n)$ are also \emph{virtual duality groups} in the sense of Bieri--Eckmann \cite{BE}. It is not known if this is the case for $\out(A_\G)$. As duality groups behave well under exact sequences \cite[Theorem~3.5]{BE}, for a positive answer it would be enough to show that every Fouxe--Rabinovitch group associated to a RAAG is a virtual duality group. This may not be true, or at least appears to be a delicate problem.

As mentioned above, we strongly suspect that $\out(A_\G)$ contains a polycyclic subgroup whose Hirsch length coincides with the vcd (this is most thoroughly discussed in a paper of Millard and Vogtmann \cite{MV}, which contains several positive results in the untwisted setting). It would also be desirable to have a closed formula for the virtual cohomological dimension in terms of (properties of) $\G$, which is something we cannot find with the recursive approach given in this paper.

\medskip

\paragraph*{\textbf{Structure of the paper.}} We describe the relevant background material on cohomological dimension and automorphism groups in Section~\ref{sec:background}. In Section~\ref{sec:vcd Fouxe-Rab} we give a proof of Theorem~\ref{th:vcdfr} and in Section~\ref{s:algorithm} describe how the decomposition series of a RORG  can be found algorithmically and complete the proof of Theorem~\ref{th:vcdrorg}.

\medskip

\paragraph*{\textbf{Acknowledgments.}} We would like to thank Conchita Mart\'inez-P\'erez and Nansen Petrosyan for very helpful conversations about cohomological dimension, and the referee for valuable feedback which helped improve the paper.

\section{Background}\label{sec:background}

\subsection{Cohomological dimension}\label{sec:cd}

For a thorough treatment of cohomological dimension, the reader is referred to the books of Bieri \cite{Bieri} and Brown \cite{Brown}.  Let $R$ be a unital commutative ring. For a group $G$, the \emph{cohomological dimension of $G$ over $R$}, denoted $\cd_R(G)$ is given by 
\[\cd_R(G)=\max \{ n : H^n(G;M)\neq 0 \textrm{ for some $RG$--module $M$}\}.  \]
The cohomological dimension of a group $G$ is given by $\cd(G)=\cd_\mathbb{Z}(G)$. The cohomological dimension satisfies $\cd_R(G)\leq \cd(G)$ for any ring $R$. A group $G$ is of \emph{finite type}, or \emph{of type $F$}, if $G$ is the fundamental group of an aspherical CW-complex with a finite number of cells. 
If $G$ is of finite type, then to find $\cd_R(G)$ one only needs to look at the cohomology with coefficients in the group ring $RG$, and  \[\cd_R(G)=\max \{ n : H^n(G;RG)\neq 0 \}.  \] If $1 \to N \to G \to Q \to 1 $ is an exact sequence of groups, then \begin{equation} \cd_R(G) \leq \cd_R(N) + \cd_R(Q). \label{leq}\end{equation} 
However, equality does not hold in general. For instance, Dranishnikov \cite{Dra1} constructed a family of hyperbolic groups $G_p$ such that $\cd(G_p)=3$ for all $p$, but $\cd(G_p\times G_q)=5$ whenever $p \neq q$. Roughly speaking, the failure of equality in \eqref{leq} comes from torsion in the top cohomology group (this is explored in detail in \cite{Dra2}). Over a field these difficulties disappear, so that one has the following:

\begin{theorem}[\cite{Bieri}, Theorem 5.5] \label{th:Bieri}
If $1 \to N \to G \to Q \to 1$ is an exact sequence of groups of finite type, then \[ \cd_{\mathbb{Q}} (G) = \cd_{\mathbb{Q}}(N) + \cd_{\mathbb{Q}}(Q). \]
\end{theorem}

Throughout this paper, we will be working with groups satisfying $\cd(G)=\cd_\mathbb{Q}(G)$, and will be able to make use of the following proposition.

\begin{proposition} \label{prop:cd-exact-sequence}
Let \[ 1 \to N \to G \to Q \to 1\] be an exact sequence of groups. Suppose that
 $N$ and $Q$ are groups of finite type, with \[ \cd_\mathbb{Q}(N)=\cd(N) \text{ and } \cd_\mathbb{Q}(Q)=\cd(Q), \] then \[\cd_\mathbb{Q}(G)=\cd(G)=\cd(N) + \cd(Q) .\]
\end{proposition}

\begin{proof}
By applying Theorem~\ref{th:Bieri} and equation \eqref{leq} we have \begin{align*} \cd(G) &\geq \cd_\mathbb{Q}(G)
\\ &=\cd_{\mathbb{Q}}(N) + \cd_{\mathbb{Q}}(Q) 
\\ &=\cd(N)+\cd(Q)
\\ &\geq \cd(G), \end{align*} so there is equality throughout.
\end{proof}

Any group with torsion has infinite cohomological dimension. However, if $G$ has a finite-index subgroup $H$ with finite cohomological dimension, then a theorem of Serre (\cite{Serre}, or alternatively \cite[VIII.3]{Brown}) asserts that for any other torsion-free finite-index subgroup $H'$ one has $\cd(H)=\cd(H')$. It follows that if $G$ contains a torsion-free subgroup of finite index, then the \emph{virtual cohomological dimension} of $G$ can be defined by \[ \vcd(G)=\{\cd(H) : H \text{ is torsion free and $[G:H] < \infty$} \}. \]

If $P$ is a torsion-free polycyclic group, then $\cd(P)=\cd_\mathbb{Q}(P)=h(P)$, where $h(P)$ is the \emph{Hirsch length} of $P$---the number of infinite cyclic factors in a normal series for $P$ (this follows from Proposition~\ref{prop:cd-exact-sequence}). If $P$ is a subgroup of a group $G$ then $\cd_R(P) \leq \cd_R(G)$, so polycyclic groups can be used to find lower bounds for (rational) cohomological dimension. In particular, one has:

\begin{proposition}\label{p:additive}
Suppose that $G$ acts properly and cocompactly on a contractible complex of dimension $d$ and contains a polycyclic subgroup $P$ with Hirsch length $h(P)=d$. Then for any finite-index, torsion-free subgroup $H$ of $G$, one has \[ cd_\mathbb{Q}(H)=\cd(H) = d.\] In particular, if $G$ has a finite-index torsion-free subgroup then $\vcd(G)=d$. \qed
\end{proposition}

Culler and Vogtmann \cite{CuV} use the spine of Outer space and the existence of a free abelian subgroup of rank $2n-3$ to show that for any torsion-free, finite-index subgroup $H$ of $\out(F_n)$ one has \[ \cd_\mathbb{Q}(H)=\cd(H)=2n-3. \] 
Similarly, by combining Borel and Serre's calculation of the vcd \cite{BorelSerre} with the upper-triangular matrices in $\gln$, we see that \[ \cd_\mathbb{Q}(H)=\cd(H)=\frac{n(n-1)}{2}, \]
for any torsion-free, finite index subgroup $H$ of $\gln$.

\subsection{RAAGs and RORGs} \label{sec:raag-background}

\subsubsection{Automorphism groups of RAAGs}

Let $A_\G$ be the right-angled Artin group determined by a finite graph $\G$.  Let us first fix names and notation for some common automorphisms:

\begin{itemize}
\item
\textbf{Graph symmetries.} 
Any automorphism of the graph induces an automorphism of the group via the corresponding permutation of the generating set. These elements of $\aut(A_\G)$ are called \emph{graph symmetries}.
\item
\textbf{Inversions.} If $v$ is a vertex of $\G$, then there is an \emph{inversion} $\iota_v$ that sends $v$ to $v^{-1}$ and fixes all other generators of $A_\G$.
\item
\textbf{Transvections.} Suppose $v$ and $w$ are distinct vertices of $\G$ with $\lk(v)\subset \st(w)$. 
There is a \emph{right transvection} $\rho^{w}_v$ which takes $v$ to $vw$ and fixes all other generators of $A_\G$. There is also a \emph{left transvection} $\lambda^{w}_v$ taking $v$ to $wv$ and fixing all other generators.
\item
\textbf{Extended partial conjugations.} Let $v$ be a vertex of $\G$ and let $K$ be a union of connected components of $\G - \st(v)$. There is an \emph{extended partial conjugation} $\pi^v_K$ which sends $w$ to $vwv^{-1}$ if $w$ is a vertex of $K$, and fixes each generator which is not a vertex of $K$.
\end{itemize}

By a theorem of Laurence \cite{Laurence}, the above automorphisms generate the whole automorphism group $\aut(A_\G)$. Given $\phi \in \aut(A_\G)$, we use $[\phi]$ or $\Phi$ to denote the outer automorphism represented by $\phi$. We will use the names of the automorphisms above to also describe their images in $\out(A_\G)$. This should be clear based on context; we will mostly be working in $\out(A_\G)$ below. 
Furthermore, we will often pass to the finite index subgroup $\out^0(A_\G)$ of $\out(A_\G)$ generated by inversions, transvections, and extended partial conjugations.

\subsubsection{Relative outer automorphism groups of RAAGs (RORGs)}

If $\Delta$ is a full subgraph of $\G$ we use $A_\Delta$ to denote the \emph{special subgroup} generated by the vertices contained in $\Delta$. An outer automorphism $\Phi$ of $A_\G$ \emph{preserves} $A_\Delta$ if there exists a representative $\phi \in \Phi$ that restricts to an automorphism of $A_\Delta$. An outer automorphism $\Phi$ \emph{acts trivially on $A_\Delta$} if there exists a representative $\phi \in \Phi$ acting as the identity on $A_\Delta$. 

\begin{definition}[RORGs]
If $\calg$, $\calh$ are collections of special subgroups of $A_\G$, then the \emph{relative outer automorphism group} (or \emph{RORG}), $\outgh$ consists of automorphisms that preserve each $A_\Delta \in \calg$ and act trivially on each $A_\Delta \in \calh$. 
\end{definition}

Similarly to the absolute case, we can define the finite-index subgroup
\[ \outogh := \outgh \cap \out^0(A_\G) \]
Laurence's theorem also extends to the relative setting:

\begin{theorem}[\cite{DW2}, Theorem~D]
The group $\outogh$ is generated by inversions, transvections, and extended partial conjugations.
\end{theorem}

For working with examples and, importantly for this paper, making a process algorithmic it is important for us to be able  to answer the following questions:

\begin{enumerate}
\item Which inversions, transvections, and partial conjugations are contained in $\outogh$? \label{q2}
\item If $A_\Delta$ is a special subgroup of $A_\G$, which inversions, transvections, and partial conjugations preserve $A_\Delta$? Which of these fix $A_\Delta$? \label{q1}
\item Given a special subgroup $A_\Delta$ (not necessarily in $\mathcal{G}$ or $\mathcal{H})$, is $A_\Delta$ preserved by $\outogh$? If so, is $A_\Delta$ invariant under $\outogh$? \label{q3}
\end{enumerate}

Note that if we can check \eqref{q1} for a generator $\Phi$ then we can also check \eqref{q2} by checking if $\Phi$ preserves every element of $\mathcal{G}$ and fixes every element of $\mathcal{H}$. Similarly, once we have a generating set for $\outogh$ we can check if another subgroup $A_\Delta$ is invariant or fixed under this group by running the check generator-by-generator. The conditions for $\eqref{q1}$ are covered by Lemma~2.2 of \cite{DW2}:

\begin{lemma}[\cite{DW2}, Lemma~2.2] \label{le:genexercise}
Let $A_\Delta$ be a special subgroup of $A_\G$ and let $[\iota_v]$, $[\rho^w_v]$, and $[\pi^v_K]$ be an inversion, transvection, and extended partial conjugation, respectively.
\begin{itemize}
 \item The inversion $[\iota_v]$ acts trivially on $A_\Delta$ if and only if $v \not \in \Delta$. It always preserves $A_\Delta$.
 \item The transvection $[\rho^w_v]$ acts trivially on $A_\Delta$ if and only if $v \not \in \Delta$. It preserves $A_\Delta$ if and only if $v \not \in \Delta$ or both $v,w \in \Delta$.
 \item The extended partial conjugation $[\pi^v_K]$ acts trivially on $\Delta$ if and only if \begin{equation} \label{eq:star} \tag{$\ast$}  K \cap \Delta = \emptyset \text{ or } \Delta - \st(v) \subset K.\end{equation} The subgroup $A_\Delta$ is preserved if and only if $\Delta$ satisfies \eqref{eq:star} or $v \in \Delta$.
\end{itemize}

\end{lemma}

Lemma~\ref{le:genexercise} is often enough for dealing with small examples in practice. We will finish this section by introducing some slightly more sophisticated language that is useful for applying Lemma~\ref{le:genexercise} (and hence answering questions \eqref{q2}--\eqref{q3}) when working with the theory in general or dealing with larger examples.

\begin{definition}[$\calj^v$, $\calj$-paths, and $\calj$-components]

Let $\calj$ be a collection of special subgroups of $A_\G$.  Given a vertex $v \in \G$ and a collection of special subgroups $\calj$ we define $\calj^v$ to be the subset of $\calj$ consisting of special subgroups that \emph{do not} contain $v$, so that:\[ \calj^v= \{A_\Delta \in \calj: v\not \in \Delta \}. \]
A \emph{$\calj$-path} in $\G$ is a sequence of vertices $v_1, \ldots,v_k$ of $\Gamma$ such that each pair $(v_i,v_{i+1})$ either span an edge of $\G$ or are contained in some common element of $\calj$. 
A \emph{$\calj$-component} of a subgraph $\Delta \subset \Gamma$ is a maximal subgraph $C \subset \Delta$ with the property that any two vertices in $C$ are connected by a $\calj$-path in $\Lambda$. Equivalently, we can obtain $\calj$-components of $\Delta$ by gluing up connected components of $\Delta$ when they both intersect the same element of $\calj$.
\end{definition}

We define the partial preorder $\leq_{(\calg,\calh)}$ on $V(\G)$ by saying that $v \leq_{(\calg,\calh)} w$ if and only if $\lk(v) \subset \st(w)$ and $v \not \in \calg^w \cup \calh$ (equivalently $[\rho^w_v] \in \outgh$).  Given a subgraph $\Delta \subset \G$ we say that $\Delta$ is \emph{upwards closed} under $\leq_{(\calg,\calh)}$ 
if $v \in \Delta$ and $v \leq_{(\calg,\calh)} w$ implies that $w \in \Delta$. For partial conjugations, Lemma~\ref{le:genexercise} implies that an extended partial conjugation $[\pi^v_K]$ is an element of $\outogh$ if and only if $K$ is a union of $\calg^v\cup\calh$ components of $\G - \st(v)$.
We say that $\Delta$ is \emph{$(\calg,\calh)$--star-separated} by a vertex $v$ if $\Delta$ intersects more than one $(\calg^v\cup \calh)$--component of $\G-\st(v)$. This is equivalent to the existence of an extended partial conjugation $[\pi_K^v] \in \outgh$ which acts on $A_\Delta$ as a non-inner automorphism. Putting all of this together with the fact that $A_\Delta$ is invariant (respectively, fixed) by every element of $\outogh$ if and only if $A_\Delta$ is invariant (respectively, fixed) by all of its generators, we have the following:

\begin{proposition}\label{p:invariant} Let $A_\Delta$ be a special subgroup of $A_\Gamma$.
\begin{itemize} \item $A_\Delta$ is invariant under $\outogh$ if and only if $\Delta$ is upwards closed under $\leq_{(\calg,\calh)}$ and $\Delta$ is not $(\calg,\calh)$--star-separated by a vertex $v \in \G - \Delta$.
 \item The group $\outogh$ acts trivially on $A_\Delta$ if and only if $v\leq_{(\calg,\calh)} w$ implies that $v=w$ for every $v \in \Delta$, the graph  $\Delta$ is not $(\calg,\calh)$--star-separated by \emph{any} vertex of $\Gamma$, and every element of $\Delta$ is contained in an element of $\mathcal{H}$.
\end{itemize}
\end{proposition}

\begin{proof}
We have $[\rho_v^w] \in \outogh$ if and only if $v \neq w$ and $v \leq_{(\calg,\calh)} w$. From Lemma~\ref{le:genexercise} it follows that $\Delta$ is upwards closed under $\leq_{(\calg,\calh)}$ if and only if $A_\Delta$ is invariant under every transvection in $\outogh$. Similarly, $A_\Delta$ is invariant under every extended partial conjugation in $\outgh$ if and only if $\Delta$ is not $(\calg,\calh)$--star-separated by a vertex $v \in \G - \Delta$. The argument for when $A_\Delta$ is fixed runs along similar lines, with the final condition ensuring that all of the inversions in $\outogh$ also fix $A_\Delta$.
\end{proof}

\section{The virtual cohomological dimension of a Fouxe-Rabinovitch group}\label{sec:vcd Fouxe-Rab}

In this section we use the relative outer space of Guirardel and Levitt to find the virtual cohomological dimension of a Fouxe-Rabinovitch group associated to a free factor decomposition of a right-angled Artin group.

\subsection{Fouxe-Rabinovitch groups and congruence subgroups}

Let $G=G_1 \ast G_2 \ast \cdots \ast G_k \ast F_m$ be a free factor decomposition of a group. We let $\calg=\{G_i\}$ and define the \emph{Fouxe-Rabinovitch group} associated to this free factor decomposition to be 
\[ \FR=\out(G;\mathcal{G}^t). \] This is the subgroup of $\out(G)$ acting trivially on each $G_i$ in the decomposition. We do not assume the free factor decomposition is maximal: it need not be the Grushko decomposition of $G$. We do, however, require that this free factor decomposition is nontrivial in the sense that $k \geq 1$ and $k+m \geq 2$. 

The level $3$ congruence subgroup of $\FR$ is defined in the same way as the subgroups of $\gl(n,\mathbb{Z})$ of the same name. It is the finite-index subgroup $\FRl$ acting trivially on $H_1(G;\mathbb{Z}/3\mathbb{Z})$. As the action of $\FR$ on each $G_i$ is trivial, this is the same as the subgroup acting trivially on $H_1(F_m;\mathbb{Z}/3\mathbb{Z})$ via the quotient map $\FR \to \out(F_m)$. If each $G_i$ and each $G_i/Z(G_i)$ is torsion-free, then so is each level $3$ congruence subgroup of $\FR$ (see \cite[Thoerem~5.2]{GL}).

\subsection{Relative Outer space}

Following the work of Guirardel and Levitt in \cite{GL2, GLDeformation}, we recall the definition of the spine of the relative outer space given by a free factor decomposition of a group. 

\begin{definition}
A \emph{Grushko $\calg$-tree} is a minimal action of $G$ on a simplicial tree $T$ with trivial edge stabilizers such that each element of $\calg$ is elliptic in $T$ and each vertex stabilizer is either trivial or conjugate to an element of $\calg$. Two Grushko $\calg$-trees $T_1$ and $T_2$ are equivalent if there is a $G$-equivariant homeomorphism $f: T_1 \to T_2$. 
\end{definition}

The set of Grushko $\calg$-trees forms a poset, where $T_1 < T_2$ if there is a ($G$-equivariant) subforest in $T_2$ which collapses to give the action of $G$ on $T_1$. The geometric realization of this poset is called the \emph{spine of relative Outer space} and we will denote it by $X_\calg$. 

\begin{figure}[ht]
\begingroup%
  \makeatletter%
  \providecommand\color[2][]{%
    \errmessage{(Inkscape) Color is used for the text in Inkscape, but the package 'color.sty' is not loaded}%
    \renewcommand\color[2][]{}%
  }%
  \providecommand\transparent[1]{%
    \errmessage{(Inkscape) Transparency is used (non-zero) for the text in Inkscape, but the package 'transparent.sty' is not loaded}%
    \renewcommand\transparent[1]{}%
  }%
  \providecommand\rotatebox[2]{#2}%
  \newcommand*\fsize{\dimexpr\f@size pt\relax}%
  \newcommand*\lineheight[1]{\fontsize{\fsize}{#1\fsize}\selectfont}%
  \ifx\svgwidth\undefined%
    \setlength{\unitlength}{127.41661132bp}%
    \ifx\svgscale\undefined%
      \relax%
    \else%
      \setlength{\unitlength}{\unitlength * \real{\svgscale}}%
    \fi%
  \else%
    \setlength{\unitlength}{\svgwidth}%
  \fi%
  \global\let\svgwidth\undefined%
  \global\let\svgscale\undefined%
  \makeatother%
  \begin{picture}(1,0.86302801)%
    \lineheight{1}%
    \setlength\tabcolsep{0pt}%
    \put(0,0){\includegraphics[width=\unitlength,page=1]{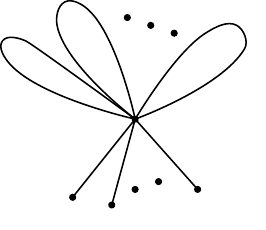}}%
    \put(0.5990318,0.38080199){\color[rgb]{0,0,0}\makebox(0,0)[lt]{\lineheight{1.25}\smash{\begin{tabular}[t]{l}$G_1$\end{tabular}}}}%
    \put(0.22138247,0.03491505){\color[rgb]{0,0,0}\makebox(0,0)[lt]{\lineheight{1.25}\smash{\begin{tabular}[t]{l}$G_2$\end{tabular}}}}%
    \put(0.38951779,0.0069409){\color[rgb]{0,0,0}\makebox(0,0)[lt]{\lineheight{1.25}\smash{\begin{tabular}[t]{l}$G_3$\end{tabular}}}}%
    \put(0.77241267,0.06259761){\color[rgb]{0,0,0}\makebox(0,0)[lt]{\lineheight{1.25}\smash{\begin{tabular}[t]{l}$G_k$\end{tabular}}}}%
  \end{picture}%
\endgroup%

\caption{A rose with $m$ petals and $k-1$ leaves, whose Bass--Serre tree gives a vertex in the spine of relative Outer space.}
\label{fig:rose}
\end{figure}
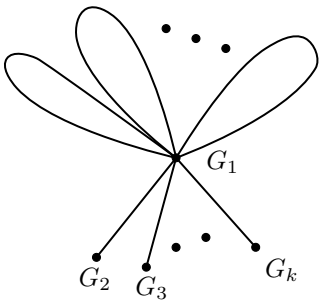

By a theorem of Guirardel and Levitt \cite{GL2}, the spine $X_\calg$ is contractible. 
The spine admits an action of $\out(G;\calg)$ by precomposing the action of $G$ on a tree $T$ with the automorphism. We will only need the restriction of this action to $\FR=\out(G;\calg^t)$. 

Each $n$-simplex corresponds to a chain $T_0< T_1 < \cdots < T_n$ of Grushko $\calg$-trees. As the action of $\FR$ preserves the number of edge orbits in a Grushko $\calg$-tree, the action of $\FRl$ on $X_\calg$ is \emph{rigid} (any automorphism preserving a simplex fixes it pointwise). The lemma below uses ideas from the proof of Proposition~3.7 of \cite{GL} and gives a description of simplex stabilizers for the action.

\begin{lemma}\label{l:simplex}
Let $\sigma$ be a simplex in $X_\calg$ corresponding to the chain $T_0 < T_1 < \cdots < T_n$ of Grushko $\calg$-trees. Then the stabilizer of $\sigma$ in $\FRl$ is a finite-index subgroup of \[ \oplus_{i=1}^k G_i^{v_i} /Z(G_i), \] where $v_i$ is the number of $G_i$-orbits  of edges at the vertex fixed by $G_i$ in $T_n$ and $Z(G_i)$ is embedded in $G_i^{v_i}$ diagonally. Furthermore, \[ \sum_{i=1}^k v_i \leq 2(m + k -1) - n. \] 
\end{lemma}

\begin{proof}
Firstly, we show that automorphisms in $\FRl$ preserving a Grushko $\calg$-tree $T$ act trivially on the quotient graph $T/G$, and therefore preserve all collapses of $T$.  
This can be seen since the leaf vertices in $T/G$ must have a non-trivial stabilizer in $G$, so must be in $\calg$. Hence the action fixes these vertices. Any such action induces a finite order element of $\out(\pi_1(T/G)) \cong \out(F_m)$. However the image of $\FRl$ in $\out(F_m)$ is also a level $3$ congruence subgroup and is torsion-free. Hence the action on $T/G$ must be trivial (note that if $T/G$ is a circle, all vertices are fixed).
As each $T_j$ is a collapse of $T_n$, any automorphism in $\FRl$ that fixes $T_n$ in $\FRl$ also fixes each $T_j$ with $j <n$, so that: 
\[ \stab_{\FRl}(\sigma) = \stab^0_{\FR}(T_n) \cap \FRl, \] 
where $\stab^0_{\FR}(T_n)$ denotes the stabilizer of  $T_n$ in $\FR$ that acts trivially on $T_n/G$. This is the \emph{group of twists} of the splitting \cite[Section 2.4]{Levitt05} and satisfies \[ \stab^0_{\FR}(T_n) \cong \oplus_{i=1}^k G_i^{v_i} /Z(G_i), \] where, as in the hypothesis, each $v_i$ is the number of $G_i$-orbits of edges in $T_n$ at the vertex fixed by $G_i$. 

It remains to justify the final inequality. Note that we have to collapse at least $n$ orbits of edges in $T_n$ to obtain $T_0$, so we may assume $T_n$ has $N \geq n$ orbits of vertices with trivial stabilizer. In total, the quotient graph $T_n/G$ has $N+k$ vertices and $N+k+m-1$ edges (as the fundamental group is $F_m)$. There are at least 3 half-edges adjacent to each of the $N$ vertices with trivial stabilizer. Subtracting these 3N half-edges from the total half-edge count gives:
\begin{align*} \sum_{i=1}^k v_i 
 &\leq  2(N+k+m-1) - 3N \\
  &=2(m + k-1) - N \\
   &\leq 2(m + k - 1) - n, \end{align*} as required. 
\end{proof}

As each $v_i \geq 1$, the inequality in Lemma~\ref{l:simplex} shows that $n \leq 2m + k -2$. It is not hard to check that the dimension of the spine is equal to $2m+k-2$ by exhibiting a graph of groups decomposition of $G$ with $2m + k -2$  trivalent vertices with trivial stabilizers, trivial edge groups, and each nontrivial vertex group corresponding to a $G_i$ (see Figure~\ref{fig:bee}). 

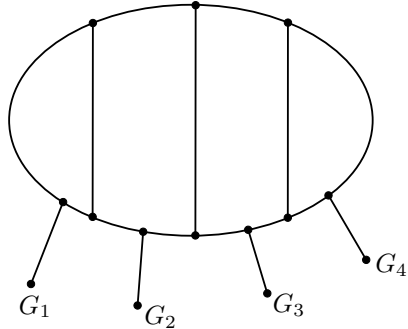
\begin{figure}[ht]
\centering
\def\svgwidth{0.6\columnwidth}
\begingroup%
  \makeatletter%
  \providecommand\color[2][]{%
    \errmessage{(Inkscape) Color is used for the text in Inkscape, but the package 'color.sty' is not loaded}%
    \renewcommand\color[2][]{}%
  }%
  \providecommand\transparent[1]{%
    \errmessage{(Inkscape) Transparency is used (non-zero) for the text in Inkscape, but the package 'transparent.sty' is not loaded}%
    \renewcommand\transparent[1]{}%
  }%
  \providecommand\rotatebox[2]{#2}%
  \newcommand*\fsize{\dimexpr\f@size pt\relax}%
  \newcommand*\lineheight[1]{\fontsize{\fsize}{#1\fsize}\selectfont}%
  \ifx\svgwidth\undefined%
    \setlength{\unitlength}{154.68480748bp}%
    \ifx\svgscale\undefined%
      \relax%
    \else%
      \setlength{\unitlength}{\unitlength * \real{\svgscale}}%
    \fi%
  \else%
    \setlength{\unitlength}{\svgwidth}%
  \fi%
  \global\let\svgwidth\undefined%
  \global\let\svgscale\undefined%
  \makeatother%
  \begin{picture}(1,0.57801325)%
    \lineheight{1}%
    \setlength\tabcolsep{0pt}%
    \put(0,0){\includegraphics[width=\unitlength,page=1]{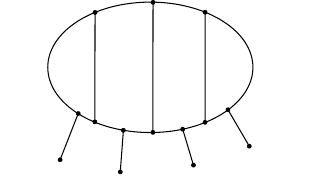}}%
    \put(0.16537063,0.03044346){\color[rgb]{0,0,0}\makebox(0,0)[lt]{\lineheight{1.25}\smash{\begin{tabular}[t]{l}$G_1$\end{tabular}}}}%
    \put(0.38528822,0.01139534){\color[rgb]{0,0,0}\makebox(0,0)[lt]{\lineheight{1.25}\smash{\begin{tabular}[t]{l}$G_2$\end{tabular}}}}%
    \put(0.6097754,0.03296453){\color[rgb]{0,0,0}\makebox(0,0)[lt]{\lineheight{1.25}\smash{\begin{tabular}[t]{l}$G_3$\end{tabular}}}}%
    \put(0.78976913,0.10398234){\color[rgb]{0,0,0}\makebox(0,0)[lt]{\lineheight{1.25}\smash{\begin{tabular}[t]{l}$G_4$\end{tabular}}}}%
  \end{picture}%
\endgroup%

\caption{A graph of groups decomposition of $G$ with a maximal number of edge orbits in the case that $k=m=4$.}
\label{fig:bee}
\end{figure}

The above work allows one to bound the geometric dimension of a Fouxe-Rabinovitch group via the following theorem:

\begin{theorem}[\cite{MR2365352}, Theorem~7.3.3] \label{t:Geoghegan}
Let $X$ be a contractible, rigid $G$-CW complex with $\dim(X) \leq N$. For each $n$, suppose that the stabilizer of each $n$-cell in $X$ has geometric dimension at most $d_n$. Then G has geometric dimension \[ \gd(G) \leq \max\{d_n + n : 0 \leq n \leq N\}. \]
\end{theorem}

We will apply this to the specific case of right-angled Artin groups below.

\subsection{Free product decompositions of RAAGs}

For a finite graph $\G$, we define $d(\Gamma)$ to be the size of the largest clique in $\G$. This is the same as the dimension of the Salvetti complex of $A_\G$. We define $z(\G)$ to be the number of vertices in $\G$ that are adjacent to every other vertex. This is the same as the rank of the center of $A_\G$, which is a finitely-generated free-abelian group. As $\FRl$ is finite-index in $\FR$, the next theorem and its corollary imply Theorem~\ref{th:vcdfr} from the introduction.

\begin{theorem}\label{th:vcd2}
Let $\FR$ be the Fouxe-Rabinovitch group associated to a nontrivial free factor decomposition  \[ A_\Gamma=A_{\Delta_1} \ast A_{\Delta_2} \ast \cdots \ast A_{\Delta_k} \ast F_m \] of a right-angled Artin group, and let $\FRl$ be its level 3 congruence subgroup. Then \[\gd(\FRl)=(k+2m-2)\cdot\max_i\{d(\Delta_i)\} + \sum_{i=1}^k (d(\Delta_i)-z(\Delta_i)). \] 
\end{theorem}

\begin{proof}
As above, let $X_\calg$ be the spine of relative Outer space and let $\sigma$ be a simplex of dimension $n$ given by the chain of trees $T_0 < T_1 \cdots < T_n$. By Lemma~\ref{l:simplex}, the stabilizer $\stab_\FRl(\sigma)$ of $\sigma$ is a finite index subgroup of \[ \oplus_{i=1}^k A_{\Delta_i}^{v_i}/Z(A_{\Delta_i}) \cong \oplus_{i=1}^k (A_{\Delta_i}^{v_i-1}\oplus A_{\Delta_i}/ Z(A_{\Delta_i})), \] and $\sum_{i=1}^k(v_i-1) \leq 2m+k-2-n.$ 
As the geometric dimension of $A_{\Delta_i}$ is $d(\Delta_i)$ and the geometric dimension of $A_{\Delta_i}/Z(A_{\Delta_i})$ is $d(\Delta_i)-z(\Delta_i)$, it follows that 
\begin{align*} \gd(\stab_\FRl(\sigma)) &\leq (2m+k-2-n)\cdot\max_i\{d(\Delta_i)\} + \sum_{i=1}^k (d(\Delta_i)-z(\Delta_i)) \\ &\leq [(2m+k-2)\cdot\max_i\{d(\Delta_i)\} + \sum_{i=1}^k (d(\Delta_i)-z(\Delta_i))] -n.\end{align*}
Therefore Theorem~\ref{t:Geoghegan} implies that \[ \gd(\FRl) \leq (2m+k-2)\cdot\max_i\{d(\Delta_i)\} + \sum_{i=1}^k (d(\Delta_i)-z(\Delta_i)). \] 
To establish equality it is enough to find a free abelian subgroup of $\FR$ of rank equal to the right hand side of this equation. If we reorder the vertices so that $A_{\Delta_1}$ has maximal dimension, we can find such a group inside the stabilizer of (the tree corresponding to) the rose given in Figure~\ref{fig:rose}. In this case, Lemma~\ref{l:simplex} tells us that the stabilizer of the $0$--cell given by the rose in $X_\calg$ under $\FRl$ is a finite index subgroup of \[ A_{\Delta_1}^{2m+k-1}/Z(A_{\Delta_1}) \oplus_{i=2}^k A_{\Delta_i}/Z(A_{\Delta_i}) \cong A_{\Delta_1}^{2m+k-2}  \oplus_{i=1}^k A_{\Delta_i}/Z(A_{\Delta_i}), \] which contains a free-abelian subgroup of the desired rank. 
\end{proof}

The following corollary is immediate from the proof:

\begin{corollary}
There exists a free abelian subgroup of rank equal to $\gd(\FRl)$, so that \[ \gd(\FRl)=\cd(\FRl)=\cd_\mathbb{Q}(\FRl). \]  \qed
\end{corollary}

\begin{remark}\label{remark:abelian_description}
The free abelian subgroup used in the proof of Theorem~\ref{th:vcd2} can be given quite explicitly. 
Firstly, order the factors of the decomposition so that $A_{\Delta_1}$ has maximal dimension and let $A_i$ be the vertices of a maximal clique in $\Delta_i$.
Let $X$ be the set of vertices generating the free factor $F_m$.
Then, take all left and right transvections $\rho_x^a$ and $\lambda_x^a$ for $x\in X$ and $a\in A_1$.
Adding the partial conjugations of the subgroups $A_{\Delta_j}$, for $j >1$, by elements of $A_1$ gives  a free abelian group of rank $(2m+k-1)\cdot d(\Delta_1)$ in $\aut(A_\G)$.

For each $i=1,\ldots,k$, and each vertex $v \in A_i$, add the partial conjugation $\pi^v_{\Delta_i}$. Such partial conjugations are trivial if $v$ is in the center of $A_{\Delta_i}$, and since a maximal clique in $\Delta_i$ must contain all vertices in the center, this gives us $d(\Delta_i)-z(\Delta_i)$ partial conjugations for each $i$.
One can check that all of automorphisms above generate a free abelian subgroup of $\aut(A_\Gamma)$ of rank 
$$(2m+k-1)\cdot d(\Delta_1) + \sum_{i=1}^k (d(\Delta_i)-z(\Delta_i)).$$

The only inner automorphisms that appear in the above group come from products of generators with acting letter $a\in A_1$, so that the intersection of this subgroup with the inner automorphisms has rank $d(\Delta_1)$. Subtracting this gives the rank in $\out(A_\G)$.
\end{remark}

\section{Calculating the vcd}\label{s:algorithm}

We now give the details of the algorithm to compute the vcd of a RORG. As a first step, we explain how the decomposition procedure for a RORG given in \cite{DW2} is algorithmic.

\subsection{Dismantling RORGs}\label{s:decomposition}

The finite-index subgroup $\outogh$ of a RORG can be broken up in the following way:

\begin{theorem}[{\cite[Theorem A]{DW2}}]\label{th:decomposition}
 The group $\outogh$ admits a subnormal series
 \[1=H_0 < H_1 < H_2 < \cdots < H_K = \outogh\]
 such that each quotient $Q_i=H_i/H_{i-1}$ is either:
 \begin{enumerate}\renewcommand{\theenumi}{$\mathcal{D}\arabic{enumi}$}
  \item\label{dec:ab} a finitely generated free abelian group,
  \item\label{dec:GL} isomorphic to $\glm$, or
  \item\label{dec:FR} a Fouxe-Rabinovitch group given by a free factor decomposition of a special subgroup of $A_\G$. \qed
 \end{enumerate}
\end{theorem}

Note that $\out(F_m)$ may arise as a quotient via case~\eqref{dec:FR}. As in the introduction, we call such a subnormal series a \emph{decomposition series} for the group.

The most natural way to find the consecutive quotients in a decomposition series is to first build a \emph{decomposition tree} for $\outogh$. This is a rooted tree where every internal vertex is labelled by a group of the form $\out^{0}(\G_v;\calg_v,\calh_v)$, with $\G_v$ a subgraph of $\G$ and $\calg_v, \calh_v$ sets of special subgroups of $A_{\G_v}$. Our initial group is at the root.  Each internal vertex $G_v$ has two descendants $K_v$ and $I_v$ forming an exact sequence 
\[ 1 \to K_v \to G_v \to I_v \to 1.\]
Every leaf of this tree is labelled by a group of the form \eqref{dec:ab}, \eqref{dec:GL}, or \eqref{dec:FR} and one can show (e.g. using induction on the size of the tree) that the leaves of the tree give consecutive quotients in a subnormal series for the root. An example of such a tree is given in \cite[Figure~6]{DW2}. 

\begin{proposition}\label{prop:algorithm for dec tree}
There is an algorithm that produces a decomposition tree for $\outogh$.
\end{proposition}

\begin{proof}
	The process for obtaining a tree is iterative. Given a vertex $v$ in the tree, labelled by $\outghv$, we describe below how to either
	\begin{enumerate}
		\item recognise $\outghv$ as a group of type \eqref{dec:ab}, \eqref{dec:GL}, or \eqref{dec:FR}, or
		\item extend the tree by adding two edges and two new descendants $v_1$ and $v_2$ of $v$, such that the group labelling $v_1$ and $v_2$ is either a RORG of lower complexity (see  \cite[Theorem 5.9]{DW2} for details on the complexity), or a group of type \eqref{dec:ab}, \eqref{dec:GL}, or \eqref{dec:FR}.
	\end{enumerate}
	If a new vertex has not been recognised as \eqref{dec:ab}, \eqref{dec:GL}, or \eqref{dec:FR}, then we repeat the process on this vertex.
	Because the complexity of RORGs decreases as we get further from the root, this algorithm will terminate.
	
	Given $\outghv$, the first step is to extend $\calg_v$ to its \emph{saturation} $\calg_v'$ relative to $(\calg_v,\calh_v)$, which is the collection $\calg_v'$ of all special subgroups that are invariant under $\outghv$.
	The invariant special subgroups can be determined from the input using Proposition~\ref{p:invariant}.
	
	Now assume that $\calg_v$ is saturated with respect to $(\calg_v,\calh_v)$. By \cite[Theorem~E]{DW2}, for each special subgroup $A_\Delta$, the image of $\outghv$ under the restriction map $R_\Delta$ is equal to $\out^{0}(A_\Delta;(\calg_v)_\Delta,(\calh_v)_\Delta^t)$,
	where \[(\calg_v)_\Delta = \{A_{\Delta \cap \Theta} \mid A_\Theta \in \calg_v \} \setminus \{ A_\Delta \},\] and $(\calh_v)_\Delta$ is defined similarly.
	This image is nontrivial if and only if there is an inversion, extended partial conjugation, or transvection with nontrivial image under $R_\Delta$. 
	This is a finite list of elements, and checking if each one has nontrivial image is a simple process. 
	
	We now divide into cases according to the nature of the images of restriction maps.
	
	\paragraph{\textbf{Case 1.}} There is a restriction map $R_\Delta$ with nontrivial image.\nopagebreak
	
	In this case we use the exact sequence
	\[
	\begin{split}
	1\to \out^{0}(A_{\G_v};\calg_v,(\calh_v\cup\{A_{\Delta}\})^t) \to  &\out^{0}(A_{\G_v};\calg_v,\calh_v^t) \\
	& \stackrel{R_{\Delta}}{\longrightarrow} \out^{0}(A_\Delta;(\calg_v)_\Delta,(\calh_v)_\Delta^t) \to 1.
	\end{split}
	\]
	given by \cite[Theorem~E]{DW2}. As per the proof of \cite[Theorem~5.9]{DW2}, the complexity of the RORG in the kernel and quotient is strictly lower than that of $\outghv$.
	In this case we add two new descendants below $v$ in the tree, with vertices labelled by the kernel and image above.
	
	\paragraph{{\bf Case 2.}} All restriction maps have trivial image. \nopagebreak
	
	As in \cite[Section~5]{DW2}, we can break into five subcases.

	\paragraph{\textbf{Case 2a.}} $\G_v$ is disconnected and $\G_v$ is $\calg_v$-disconnected.\nopagebreak
	
	Here  $\outghv$ is a Fouxe-Rabinovitch group where $F_m$ is a free group on $m$ isolated vertices not contained in any element of $\calg_v$ and the $A_{\Delta_i}$ are  the remaining $\calg_v$--connected components (\cite[Proposition~5.2]{DW2}).

	\paragraph{\textbf{Case 2b.}} $\G_v$ is disconnected and $\G_v$ is $\calg_v$-connected.\nopagebreak
	
	The vertices which $(\calg_v,\calh_v)$-star-separate form a complete graph $\Theta$ and, as per the proof of \cite[Proposition~5.2]{DW2}, $\outghv$ is a free-abelian group of rank equal to $|\Theta|$. 	
	
	\paragraph{\textbf{Case 2c.}} $\G_v$ is connected and the center $Z(A_{\G_v})$ of $A_{\G_v}$ is trivial.\nopagebreak
	
	In this case $\outghv$ is generated by commuting partial conjugations with acting letters $v$ that have $N_v$ $(\calg,\calh)$-connected components. It is not hard to check (e.g. using the first Johnson homomorphism \cite{W1}) that these elements form a free-abelian group of rank $\sum (N_v-1)$. 
	
	\paragraph{\textbf{Case 2d.}} $\G_v$ is connected and $Z(A_{\G_v})$ is a proper, nontrivial subgroup.\nopagebreak
	
	If $\Delta=\Gamma_v - Z(\Gamma_v)$ we apply \cite[Proposition~5.6]{DW2}.
	There is a projection homomorphism $P_\Delta$ with image $\out^{0}(A_\Delta;(\calg_v)_\Delta^t)$ (with $(\calg_v)_\Delta$ as defined above) whose kernel is a free abelian group with basis given by the \emph{leaf transvections} in $\outogh$.
	These are transvections $\rho^u_w$ with $w\in Z(\G)$ and $u \notin Z(\G)$, see \cite{CVSubQuot}. 
	We therefore add two descendants below $v$, one labelled by a free abelian group of the appropriate rank, and the other labelled by $\out^{0}(A_\Delta;(\calg_v)_\Delta^t)$.

	\paragraph{\textbf{Case 2e.}} $\G_v$ is complete and $A_{\G_v}=\mathbb{Z}^n$ for some $n$.\nopagebreak
	
	It is described in \cite[Proposition~5.8]{DW2} how the group fits in the exact sequence
	\[ 1\to A \to \outghv \to \glm \to 1, \] where $A$ is a finitely generated free abelian group of matrices, so that the rank is easy to compute. 
	We thus add two descendants below $v$, one labelled by $A$ and the other by $\glm$.	
\end{proof}

Note that the construction of a decomposition tree involves many choices, as at each step we only pick \emph{some} invariant special subgroup $A_\Delta$ for which there is a restriction map. 

\begin{question}
Does the set of consecutive quotients in a decomposition series depend on the set of choices made to dismantle $\outogh$? 
\end{question}

In this direction, Br\"uck \cite[Section 7]{Brueck} uses careful choices of restriction maps to construct a decomposition tree for $\outogh$ where the leaves can be described quite explicitly. As a trade-off, the leaves that appear in the decomposition tree of Br\"uck are slightly more general (there are groups generated by partial conjugations that are not necessarily of type \eqref{dec:ab}, \eqref{dec:GL}, or \eqref{dec:FR}).

\subsection{Completing the proof of Theorem~B} 

To complete the proof of Theorem~B, we describe how to compute the vcd of a RORG step-by-step:
\medskip

\paragraph*{\textbf{Step 1: Build a decomposition tree for $\outogh$.}} 

This is detailed in  Proposition~\ref{prop:algorithm for dec tree}.
\medskip

\paragraph*{\textbf{Step 2: Find the vcd of each leaf.}}

Each leaf is free-abelian, a copy of $\gl(n,\mathbb{Z})$, or a Fouxe-Rabinovitch group, so this can be read off via the calculations of Borel--Serre \cite{BorelSerre} and Culler--Vogtmann \cite{CuV} discussed in Section~\ref{sec:cd} and Theorem~\ref{th:vcdfr}.
\medskip

\paragraph*{\textbf{Step 3: Add the vcds of the leaves to find the vcd of the root}}

We do not need to explain how to carry out this step, but we should justify why it works. This is where the discussion of rational cohomological dimension given in Section~\ref{sec:cd} comes into play. The key point here is that we can restrict to the congruence subgroup $\outghp$ of $\outogh$. This is the torsion-free, finite-index subgroup given by the elements acting trivially on $H_1(A_\G,\mathbb{Z}/3\mathbb{Z})$. By \cite[Theorem~4.8]{DW2}, the short exact sequence \[
	\begin{split}
	1\to \out^{0}(A_{\G_v};\calg_v,(\calh_v\cup\{A_{\Delta}\})^t) \to  &\out^{0}(A_{\G_v};\calg_v,\calh_v^t) \\
	& \stackrel{R_{\Delta}}{\longrightarrow} \out^{0}(A_\Delta;(\calg_v)_\Delta,(\calh_v)_\Delta^t) \to 1,
	\end{split}
	\] coming from each projection map restricts to a short exact sequence 
	\[\begin{split}
	1\to \out^{[3]}(A_{\G_v};\calg_v,(\calh_v\cup\{A_{\Delta}\})^t) \to  &\out^{[3]}(A_{\G_v};\calg_v,\calh_v^t) \\
	& \stackrel{R_{\Delta}}{\longrightarrow} \out^{[3]}(A_\Delta;(\calg_v)_\Delta,(\calh_v)_\Delta^t) \to 1.
	\end{split}	\] for congruence subgroups. Similar behaviour happens with the projection maps that appear in Case 2d and Case 2e during the construction of the decomposition tree (one can see this as both of the projection maps split). As a result, one obtains an analogous decomposition tree for $\outghp$ where each vertex is a level three congruence subgroup of the corresponding vertex in the decomposition tree for $\outogh$. This gives a subnormal series
	 \[1=H_0 < H_1 < H_2 < \cdots < H_K = \outghp \] of $\outghp$ where the consecutive quotients are congruence subgroups of the leaves of the decomposition tree for $\outogh$ (and the leaves given by free-abelian groups are still free-abelian of the same rank). Some leaves, in particular those isomorphic to $\gl(1,\mathbb{Z})=\mathbb{Z}/2\mathbb{Z}$, will now be trivial. All of these groups are of finite type, and have rational cohomological dimension equal to their cohomological dimension (using either the discussion in Section~\ref{sec:cd} or Theorem~\ref{th:vcdfr}). By Bieri's theorem (Theorem~\ref{th:Bieri}) and Proposition~\ref{prop:cd-exact-sequence}, the sum of the (rational) cohomological dimensions of the leaves is equal to the cohomological dimension of $\outghp$, justifying the calculation of the vcd of $\outogh$ given above.

	 \begin{remark}\label{rem:aut-out}
	  The above work shows that the rational cohomological dimension of a (relative) outer automorphism group is the same as its cohomological dimension. As the inner automorphisms are isomorphic to $A_\G/Z(A_\G)$, the same is true for $\inn(A_\G)$. Bieri's theorem implies that the vcds of $\out(A_\G)$ and $\aut(A_\G)$ differ by the dimension of $A_\G/Z(A_\G)$.
	 \end{remark}

\bibliography{vcdbib}
\bibliographystyle{abbrv}

\end{document}